\documentclass[10pt]{amsart}
\usepackage{amsfonts,amssymb,amscd,amsmath,enumerate,verbatim,calc,amsthm,color}
\input xy
\xyoption{all}

\headsep=1cm \numberwithin{equation}{section}
\newtheorem{thm}{Theorem}[section]
\newtheorem{cor}[thm]{Corollary}
\newtheorem{lem}[thm]{Lemma}
\newtheorem{prop}[thm]{Proposition}
\newtheorem{defn}[thm]{Definition}

\newtheorem{exam}[thm]{Example}
\newtheorem{rem}[thm]{Remark}

\newcommand{\Ann}{\mbox{Ann}\,}
\newcommand{\coker}{\mbox{Coker}\,}
\newcommand{\Hom}{\mbox{Hom}\,}
\newcommand{\Ext}{\mbox{Ext}\,}
\newcommand{\Tor}{\mbox{Tor}\,}
\newcommand{\Spec}{\mbox{Spec}\,}

\newcommand{\Ass}{\mbox{Ass}\,}

\newcommand{\Supp}{\mbox{Supp}\,}

\newcommand{\gr}{\mbox{grade}\,}
\newcommand{\rgr}{\mbox{r.grade}\,}
\newcommand{\depth}{\mbox{depth}\,}
\renewcommand{\dim}{\mbox{dim}\,}

\newcommand{\syz}{\mbox{syz}\,}

\newcommand{\Tr}{\mbox{Tr}\,}

\newcommand{\pd}{\mbox{Pd}\,}
\newcommand{\gc}{\mbox{G}_{C}\mbox{-dim}\,}

\newcommand{\gkd}{\mbox{G}_{K}\mbox{-dim}\,}
\newcommand{\gk}{\mbox{G}_{K}\mbox{-}\,}

\newcommand{\gd}{\mbox{G-dim}\,}

\newcommand{\h}{\mbox{ht}\,}

\renewcommand{\H}{\mbox{H}}

\newcommand{\fa}{\mathfrak{a}}
\newcommand{\fb}{\mathfrak{b}}

\newcommand{\fm}{\mathfrak{m}}
\newcommand{\fp}{\mathfrak{p}}

\newcommand{\fc}{\mathfrak{c}}

\begin{document}
\bibliographystyle{amsplain}


\title[Linkage of finite Gorenstein dimension modules]
 {Linkage of finite Gorenstein dimension modules}

\bibliographystyle{amsplain}

     \author[M. T. Dibaei]{Mohammad T. Dibaei$^{1}$}
     \author[A. Sadeghi]{Arash Sadeghi$^2$}

\address{$^{1, 2}$ Faculty of Mathematical Sciences and Computer,
Tarbiat Moallem University, Tehran, Iran.}

\address{$^{1, 2}$ School of Mathematics, Institute for Research in Fundamental Sciences (IPM), P.O. Box: 19395-5746, Tehran, Iran }
\email{dibaeimt@ipm.ir}
\email{sadeghiarash61@gmail.com}

\keywords{Linkage of modules, Gorenstein dimension, reduced grade, Semidualizing module, $\gk$dimension.\\
1. M.T. Dibaei was supported in part by a grant from IPM (No. 90130110)}
\subjclass[2000]{}

\maketitle
\begin{abstract}
 For a horizontally linked module, over a commutative semiperfect Noetherian ring $R$, the connections of its invariants reduced grade,
  Gorenstein dimension and depth are studied. It is shown that under certain conditions the depth of a horizontally linked module is equal
   to the reduced grade of its linked module. The connection of the Serre condition $(S_n)$ on an $R$--module of finite Gorenstein
   dimension with the vanishing of the local cohomology groups of its linked module is discussed.

\end{abstract}

\section{Introduction}
The theory of linkage of algebraic varieties introduced by Peskine and Szpiro \cite{PS}. In \cite{MS}, Martsinkovsky and Strooker  give its
 analogous definition for modules over non--commutative semiperfect Noetherian rings by using the composition of the two
 functors:
 transpose and syzygy. These functors and their compositions were studied by Auslander and Bridger in \cite{AB}.
The Gorenstein (or $G$-) dimension was introduced by Auslander \cite{A1} and studied by Auslander and Bridger in \cite{AB}. In this
paper, we study the theory of linkage for class of modules which have finite Gorenstein dimensions. In particular, for a horizontally
linked module $M$ of finite and positive  $G$-dimension, we study
the role of its reduced grade,  $\rgr(M)$, on the depth of its linked module $\lambda M$.

In section 2, after preliminary notions and definitions, the
associated prime ideals
 of $\Ext^{\tiny{\rgr(M)}}_R(M,R)$ is determined in terms of $\lambda M$
 for any
horizontally linked module $M$. As a consequence, we determine the
reduced grade of a
 horizontally linked module of finite Gorenstein dimension. Also a characterization of a horizontally linked module to have $G$-dimension
  zero is given in Proposition \ref{Gd}. In Theorem \ref{t} , we give some equivalent conditions for a horizontally linked module $M$ of finite and positive $G$-dimension such that $\depth$ of $M$ is equal to the reduced grade of $\lambda M$.

In section 3, we study the reduced $G$-perfect modules (Definition \ref{rgr}).
For a reduced $G$-perfect module $M$ of $G$-dimension $n$ over a
Cohen-Macaulay local
ring $R$ of dimension $d$, it is shown that
$\depth_R(M)+\depth_R(\lambda M)=d+\depth_R(\Ext^n_R(M,R))$(Theorem \ref{d}). We also generalize this formula for modules of finite
$G$-dimensions (not necessarily reduced $G$-perfect) with some conditions on depth of extension modules (Theorem \ref{t1}).
In Proposition \ref{P1}, we establish a characterization for stable reduced $G$-perfect module to be horizontally linked. More precisely,
for a stable reduced $G$-perfect module $M$ of $G$-dimension $n$, it is shown that $M$ is horizontally linked if and only if
$\rgr(M)+\rgr(\lambda M)=\gr_R(\Ext^n_R(M,R))$.

In section 4, we generalize \cite[Theorem 4.1]{Sc}. More precisely,
for horizontally linked module $M$ of finite $G$-dimension over
Cohen-Macaulay local ring of dimension $d$, it is shown that $M$
satisfies the Serre condition $(S_k)$ if and only if
$\H^i_\fm(\lambda M) =0$ for all $i$, $d-k<i<d$.

Section 5 is dedicated to study the class of modules which are evenly linked by $G_K$-Gorenstein ideals, where $K$ is a semidualizing
module. We first generalize \cite[Propositions 15]{MS} and \cite[Propositions 16]{MS}. For a semidualizing module $K$  and $\gk$Gorenstein
 ideals $\fc_1$, $\fc_2$ of grade $n$ over local ring $R$, it is shown that the $\Ext^i_R(M_1,K)\cong\Ext^i_R(M_2,K)$ for all $i$, $i>n$,
 if there is an $R$--module $M$ such that
 $M_1\underset{\fc_1}\thicksim M$ and $M\underset{\fc_2}\thicksim M_2$ (Definition \ref{def1}, Proposition \ref{p4}).
As a consequence, $\gk$dimension is preserved under evenly linkage (Corollary \ref{end}).\\


Throughout the paper, $R$ is a commutative semiperfect Noetherian ring and all modules are finite (i.e. finitely generated)
$R$--modules so that any such module has a projective cover.

Let
$P_1\overset{f}{\rightarrow}P_0\rightarrow M\rightarrow 0$ be a
finite projective presentation of $M$. The transpose of $M$, $\Tr M$,
is defined to be $\coker f^*$, where $(-)^* := \Hom_R(-,R)$, which satisfies in the exact sequence
\begin{equation} \label{n1}
0\rightarrow M^*\rightarrow P_0^*\rightarrow P_1^*\rightarrow \Tr M\rightarrow 0
\end{equation}
and is unique up to projective equivalence. Thus the minimal projective
presentations of $M$ represent isomorphic transposes of $M$. Here is a straightforward application of $\Tr M$, which is
quoted from \cite[Proposition 6.3]{A}. For an $R$--module $M$ the evaluation map $e_M:M\rightarrow M^{**}$ has kernel and cokernel
isomorphic to $\Ext^1_R(\Tr M,R)$ and $\Ext^2_R(\Tr M,R)$, respectively. More precisely, one has the following exact sequence
\begin{equation}\label{m}
0\longrightarrow \Ext^1_R(\Tr M,R)\longrightarrow M\overset{e_M}{\longrightarrow} M^{**}
\longrightarrow\Ext^2_R(\Tr M,R)\longrightarrow0.
\end{equation}

Let $P\overset{\alpha}{\rightarrow}M$ be an epimorphism such that $P$ is
a projective. The syzygy module of $M$, denoted by $\Omega M$, is
the kernel of $\alpha$ which is unique up to projective equivalence.
Thus $\Omega M$ is uniquely determined, up to isomorphism, by a
projective cover of $M$.

Recall that, over a Gorenstein local ring $R$, two ideals $\fa$ and $\fb$  are said to be linked by a Gorenstein ideal $\fc$ if
 $\fc\subseteq\fa\cap\fb$, $\fa=\fc:_R\fb$ and $\fb=\fc:_R\fa$. Equivalently, the ideals
$\fa$ and $\fb$ are linked by $\fc$ if and only if the
ideals $\fa/\fc$ and $\fb/\fc$ of the ring $R/\fc$ are linked by the zero ideal of $R/\fc$. To define
linkage for modules,
Martsinkovsky and Strooker introduced the operator $\lambda=\Omega\Tr$.  In \cite[Proposition 1]{MS}, it is shown that ideals $\fa$
and $\fb$ are linked by zero ideal if and only
if $R/\fa$ and $R/\fb$ are related to each other through the operator $\lambda$; more precisely, $R/\fa\cong\lambda (R/\fb)$ and
$R/\fb\cong\lambda (R/\fa)$.
\begin{defn}\cite[Definition 3]{MS} \emph{
Two $R$--modules $M$ and $N$ are said to
be\emph{ horizontally linked} if $M\cong \lambda N$ and
$N\cong\lambda M$. Also, $M$ is called horizontally linked (to $\lambda M$)
if $M\cong\lambda^2M$.}
\end{defn}
Having defined the horizontal linkage, the general linkage for modules is defined as follows.
\begin{defn}\label{def1}
\cite[Definition 4]{MS} \emph {An $R$--module
$M$ is said to be \emph{linked} to an $R$--module $N$, by an ideal $\fc$ of $R$, if $\fc\subseteq\Ann_R(M)\cap\Ann_R(N)$ and M and
 N are horizontally
linked as $R/{\fc}$--modules. In this situation we denote $M\underset{\fc}{\thicksim}N$.}
\end{defn}
Two modules $M$ and $N$ are called \emph{stably isomorphic} and write $\underline{M}\cong\underline{N}$ if $M\oplus P\cong N\oplus Q$
for some projective
modules $P$ and $Q$. A \emph{stable} module is a module with no non-zero projective direct summands.

Let $M$ be a stable $R$--module and let $P_1\rightarrow P_0\rightarrow M\rightarrow 0$ be a minimal projective presentation of $M$.
 Then
$P_0^*\rightarrow P_1^*\rightarrow \Tr M\rightarrow 0$ is a minimal projective presentation of $\Tr M$ (see \cite[Theorem 32.13]{AF}).
We quote the following induced exact sequences
\begin{equation}\label{2}
0\longrightarrow M^*\longrightarrow P_0^*\longrightarrow \lambda M\longrightarrow0,
\end{equation}
\begin{equation}\label{1}
0\longrightarrow \lambda M\longrightarrow P_1^*\longrightarrow \Tr M\longrightarrow0.
\end{equation}

Here is a characterization of horizontally linked modules in terms of kernel of the evaluation map $e_M$
 \cite[Theorem 2 and Proposition 3]{MS}.
\begin{thm}\label{MS}
An $R$--module $M$ is horizontally linked if and only if it is stable and \emph{$\Ext^1_R(\Tr M,R)=0$}.
\end{thm}
The Gorenstein dimension was introduced by Auslander \cite{A1}, and developed by Auslander and Bridger in \cite{AB}.
\begin{defn}
An $R$--module $M$ is said to belong to the \emph{$G$}-class, $G(R)$, whenever
\begin{itemize}
            \item[(i)]{\emph{the biduality map $M\rightarrow M^{**}$ is an isomorphism;}}
            \item[(ii)]{\emph{$\Ext^i_R(M,R)=0$ for all $i>0$;}}
            \item[(iii)]{\emph{$\Ext^i_R(M^*,R)=0$ for all $i>0$.}}
\end{itemize}
\end{defn}
From the exact sequences (\ref{m}), (\ref{2}) and (\ref{1}) it is easy to see that $M$ is in $G(R)$ if and only if
$\Ext^i_R(M,R)=0=\Ext^i_R(\Tr M,R)$ for all $i>0$.
The class $G(R)$ is quite large, it contains all finitely generated projective modules.
Note that if $\underline{M}\cong\underline{N}$ then $M$ is in $G(R)$ if and only if $N$ is in $G(R)$.
In general, since $\underline{M}\cong\underline{\Tr\Tr M}$, it follows that $M$ is in $G(R)$ if and only if $\Tr M$ is in $G(R)$
(see \cite[Lemma 4.9]{AB}), and so if $M$ is a stable $R$--module then $M$ is in $G(R)$ if and only if $\lambda M$ is in $G(R)$
and $M$ is horizontally linked (c.f. \cite[Theorem 1]{MS}).

Trivially any $R$--module $M$ has a G-resolution which is a right acyclic complex of modules in $G(R)$ whose $0$th
homology module is $M$. The module $M$ is said to have finite G-dimension, denoted by $\gd_R(M)$, if
it has a G-resolution of finite length. Note that $\gd_R(M)$ is bounded above by the projective dimension, $\pd_R(M)$, of $M$ and if
$\pd_R(M)<\infty$ then the equality holds.

For module $M$ of finite Gorenstein dimension, $\gd_R(M)$ can be expressed as follows.
\begin{thm}\cite[Theorem 29 and Lemma 23]{M1}\label{Gor}
For an $R$--module $M$ of finite G-dimension, the following statements hold true.
\begin{itemize}
       \item[(i)] \emph{$\gd_R(M)=\sup\{i\geq0\mid\Ext^i_R(M,R)\neq0\}$,}
        \item[(ii)] If $R$ is local, then \emph{$\gd_R(M)=\depth R-\depth_R(M)$.}
\end{itemize}
\end{thm}
Note that if $R$ is a Gorenstein local ring then every $R$--module has finite Gorenstein dimension (see \cite[Theorem 4.20]{AB}). Thus
 Theorem \ref{Gor} implies that over Gorenstein local ring $R$, $G(R)$ is precisely the class of all maximal Cohen-Macaulay modules.

An $R$--module $M$ is called a \emph{syzygy module} if it is embedded in a projective $R$--module.
Let $i$ be a positive integer, the module $M$ is said to be an $i$th syzygy of an $R$--module $N$ (abbreviated as an $i$th syzygy) if
there exists an exact sequence
$$0\rightarrow M\rightarrow P_{i-1}\rightarrow\cdots\rightarrow P_0\rightarrow N\rightarrow 0$$ with the $P_0,\cdots,P_{i-1}$ are
projective. By convention, every module is a $0$th syzygy. In a
pioneer work, Hyman Bass \cite{B} defined reduced grade of a module
to describe modules which are an $i$th syzygies for some $i$. We
recall the definition of reduced grade from \cite[\S8]{B}.

The \emph{reduced grade} of an $R$--module $M$ is defined to be
$$\rgr(M)=\inf\{i>0\mid \Ext^i_R(M,R)\neq0\}.$$

 Note that
$\gr_R(M)=\rgr_R(M)$ if $\gr_R(M)>0$. Moreover, if $\gd_R(M)=0$ then
$\rgr(M)=\infty$. For modules of finite and positive G-dimension,
one has $\rgr(M)\leq\gd_R(M)$ and so it is finite.

Bass shows that, ``\emph{for a given positive integer $r$, a module
$H$ over a Gorenstein ring $R$ is an $r$th syzygy if and only if $H$
is reflexive and $\emph\rgr_R(H^*)\geq r-1$"} \cite[Theorem 8.2]{B}.

We are going to discuss about the reduced grade of a module which is closely related to the property denoted by
$\widetilde{S}_k$.
\begin{defn}\cite{M1}\label{S}
\emph{An $R$--module $M$ is said to satisfy the property $\widetilde{S}_k$ if $\depth_{R_\fp} (M_\fp) \geq $ min$\{ k, \depth R_\fp\}$,
for all
$\fp\in\Spec R$.}

\end{defn}
Note that if $R$ is a Cohen-Macaulay local ring and $M$ is
horizontally linked module, then $M$ satisfies $\widetilde{S}_k$ if
and only if $M$ satisfies the Serre condition $(S_k)$. Here is a
characterization for a module $M$ to satisfy $\widetilde{S}_k$ (see
\cite[Theorem 4.25]{AB} or \cite[Theorem 42]{M1}).
\begin{thm}\label{AS}
Let $k$ be a positive integer. Consider the following statements.
\begin{itemize}
           \item[(i)] \emph{$\rgr(\Tr M)>k$};
           \item[(ii)] $M$ is a $k$th syzygy;
           \item[(iii)] for each $\fp\in\emph{\Spec}(R)$, every $R_\fp$-regular sequence of length at most $k$ is also $M_\fp$-regular
           sequence;
            \item[(iv)] $M$ satisfies $\widetilde{S}_k$.
             \\Then we have the following
             \item[(a)] (i)$\Rightarrow$(ii)$\Rightarrow$(iii)$\Rightarrow$(iv);
             \item[(b)] If \emph{$\gd_R(M) < \infty$}, then $(iv)$ implies $(i)$;
             \item[(c)] If \emph{$\gd_R(M) < \infty$}, then all the statements (i)-(iv) are equivalent to $$\emph\gr_R(\emph\Ext^i_R(M,R))
             \geq i+k \text{
             for all}\  i>0.$$
          \end{itemize}
\end{thm}

Note that when $\gd_R(M)<\infty$, the above theorem of Auslander and
Bridger gives a complete description of the module $M$ to be a $k$th
syzygy, over an arbitrary ring. As any finite module over a
Gorenstein ring has finite G--dimension, this theorem may be
considered as generalization of the Bass result (see
acknowledgement).

As an immediate consequence of Theorem \ref{AS}, we have the following.
\begin{cor}\label{A}
Let $M$ be a stable $R$--module of finite G-dimension. The following statements hold true.
\begin{itemize}
\item[(i)] For a positive integer $k$, $M$ satisfies $\widetilde{S}_k$ if and only if
\emph{$\rgr(\lambda M)\geq k$} and $M$ is horizontally linked.
\item[(ii)] Assume that $M$ is a horizontally linked. Then \emph{$\gd_R(M)\neq0$} if and only if \emph{$\rgr(\lambda M)<\infty$}.
\end{itemize}

\end{cor}
\begin{proof}
(i). By Theorem \ref{AS}, $M$ satisfies $\widetilde{S}_k$ if and
only if $\Ext^i_R(\Tr M,R)=0$ for all $i$, $1\leq i\leq k$. By
Theorem \ref{MS}, it is equivalent to have $M$ horizontally linked
and $\rgr(\lambda M)\geq k$ by using (\ref{1}).

(ii). Let $\gd_R(M)\neq0$. There is $\fp\in\Spec R$, such that $\gd_{R_\fp}(M_\fp)\neq0$.
Assume, contrarily, that $\rgr(\lambda M)=\infty$ so that $M$ satisfies $\widetilde{S}_k$ for each integer $k$. Choose an integer $k$
with $k>\depth R_\fp$. We get
$\depth_{R_\fp}(M_\fp)\geq\depth R_\fp$ which yields  $\gd_{R_\fp}(M_\fp)=0$. The other side is obvious.
\end{proof}
For $k>0$, the composed functor $\mathcal{T}_k:=\Tr\Omega^{k-1}$ is
introduced by Auslander and Bridger in \cite{AB}, and recently used
by Iyama \cite{I} to define a duality between some full
subcategories (e.g. Example \ref{exa}). As we will use the composed
functor $\mathcal{T}_k$ frequently, we note some of its properties
in the following.
\begin{rem}\label{R}
\emph{(See \cite[Lemma 4]{MS} and (\ref{m}).) For a stable
$R$--module $N$, there is an exact sequence
\begin{equation}\label{e1}
0\longrightarrow \Ext^1_R(\Tr N,R)\longrightarrow
N\longrightarrow \lambda ^2N\longrightarrow 0.
\end{equation}
Note that, as the transpose of every R-module is either stable or
zero, $\Tr\mathcal{T}_kM$ is stably isomorphic to $\Omega^{k-1}M$
for every $R$--module $M$ and all $k > 0$. As
$\Ext^1_R(\Omega^{k-1}M,R)\cong \Ext^k_R(M,R)$, we obtain
\begin{equation}\label{e} 0\longrightarrow \Ext^k_R(M,R)\longrightarrow \mathcal{T}_kM\longrightarrow \lambda^2
\mathcal{T}_kM\longrightarrow 0 \ \ \text{for \ all}\ k>0,
\end{equation}
by replacing $N$ by
  $\mathcal{T}_kM$ in the above exact sequence .}
\end{rem}
\section{linkage and the reduced grade of modules}
 We start this section by expressing the associated primes of the $\Ext^{\small{\rgr(M)}}_R(M,R)$ for a horizontally linked module $M$
 of finite and positive G-dimension in terms of $\lambda M$.
\begin{lem}\label{l3}
Let $M$ be a horizontally linked $R$--module of finite and positive $G$-dimension. Set \emph{$n=\rgr(M)$}. Then
\begin{center}\small{\emph{$\Ass_R(\Ext^n_R(M,R))=\{\fp\in\Spec R\mid\gd_{R_\fp}(M_\fp)\neq0, \depth_{R_\fp}((\lambda M)_\fp)
=n=\rgr_{R_\fp}(M_\fp)\}.$}}\end{center}

\end{lem}
\begin{proof}
We first give some general comments. Let $\fp\in\Spec R$ with
$\gd_{R_\fp}(M_\fp)\not =0$ so that $\depth R_\fp>0$. As $M$ is a
syzygy, we get $\depth_{R_\fp}( M_\fp) \not= 0$. Moreover assume
that $\rgr_{R_\fp}(M_\fp)= n$. Consider the exact sequence (\ref{e})
for $M_\fp$ and $i>0$,
\begin{equation}\label{d1}
 0\longrightarrow\Ext^i_{R_\fp}(M_\fp,R_\fp)\longrightarrow \mathcal{T}_iM_\fp\longrightarrow\lambda^2_{R_\fp}
\mathcal{T}_iM_\fp\longrightarrow 0.
 \end{equation}
 Note that $\lambda^2_{R_\fp}\mathcal{T}_iM_\fp$ is stably isomorphic to $\Omega_{R_\fp}\mathcal{T}_{i+1}M_\fp$.

As $\rgr_{R_\fp}(M_\fp) = n$, $\Ext_{R_\fp}^i(M_\fp, R_\fp)= 0$ for
all $i$, $1\leq i< n$. Thus $\mathcal{T}_i M_\fp\cong
\lambda_{R_\fp}^2\mathcal{T}_i M_\fp$ for all $i$, $1\leq i<n$, by
(\ref{d1}). Therefore $\mathcal{T}_iM_\fp$ is stably isomorphic
 to $\Omega_{R_\fp}\mathcal{T}_{i+1}M_\fp$ for $1\leq i<n$. Now, inductively, we get $\underline{\lambda_{R_\fp}M_\fp}\cong\underline
 {\Omega_{R_\fp}\mathcal{T}_1M_\fp}\cong\cdots
 \cong\underline{\Omega^n_{R_\fp}\mathcal{T}_nM_\fp}$.\\

Let $\fp\in\Ass_R(\Ext^n_R(M,R))$ so that $\gd_{R_\fp}(M_\fp)\neq0$ and $\rgr_{R_\fp}(M_\fp)=n$. Therefore, by Theorem \ref{Gor},
$\depth R_\fp>\gd_{R_\fp}(M_\fp)\geq\rgr_{R_\fp}(M_\fp)=n$.
If $(\lambda M)_\fp=0$ then
$(\Tr M)_\fp$ is a free $R_\fp$--module by (\ref{1}) and, as $\underline{(\Tr_R M)_\fp}\cong\underline{\Tr_{R_\fp}M_\fp}$,
 $\Tr_{R_\fp}M_\fp$ is free and so it is equal to zero. Therefore $M_\fp$ is free $R_\fp$--module
which is a contradiction. Thus $\fp\in\Supp_R(\lambda M)$.

Since $\fp R_\fp\in\Ass_{R_\fp}(\mathcal{T}_nM_\fp)$ it follows $\depth_{R_\fp}(\mathcal{T}_n M_\fp)=0$
 and so $\depth_{R_\fp}(\lambda_{R_\fp}M_\fp)=n$. As $\underline{\lambda_{R_\fp}M_\fp}\cong\underline{(\lambda M)_\fp}$
and $n<\depth R_\fp$, we conclude that $\depth_{R_\fp}((\lambda M)_\fp)=n$.

Conversely, let $\fp\in\Spec R$ with $\gd_{R_\fp}(M_\fp)\neq0$ and
$\depth_{R_\fp}((\lambda M)_\fp)= n= \rgr_{R_\fp}(M_\fp)$.
Therefore, as mentioned above, we have
$\underline{\lambda_{R_\fp}M_\fp}\cong\underline{\Omega^n_{R_\fp}
\mathcal{T}_nM_\fp}$ and also
$\depth_{R_\fp}(\lambda_{R_\fp}M_\fp)=\depth_{R_\fp}((\lambda
M)_\fp)=n$. Hence, it is easy to see that
$\depth_{R_\fp}\mathcal{T}_nM_\fp=0$. Consider the exact sequence
(\ref{d1}) for $i= n$. As $\lambda^2_{R_\fp}\mathcal{T}_nM_\fp$ is a
syzygy as $R_\fp$--module and $\depth R_\fp\not = 0$ we find that
$\depth_{R_\fp} (\Ext^n_{R_\fp}(M_\fp,R_\fp))=0$  and so
$\fp\in\Ass_R(\Ext^n_R(M,R))$.
\end{proof}

We can express the reduced grade of a horizontally linked module $M$ with finite G-dimension in terms of $\lambda M$. 

\begin{prop}\label{p}
Let $M$ be a horizontally linked $R$--module of finite G-dimension.
Then \emph{$$\rgr(M)=\inf\{\depth_{R_\fp}((\lambda M)_\fp)\mid
\fp\in\Spec R,\gd_{R_\fp}(M_\fp)\neq0\}.$$}
\end{prop}
\begin{proof}
We may assume that $\gd_R(M)>0$. Set $n=\rgr(M)$. Let $\fp\in\Spec R$ with $\gd_{R_\fp}(M_\fp)\neq0$ so that $\depth R_\fp>0$.
As $M$ is a syzygy, we get $\depth_{R_\fp}( M_\fp)>0$. Therefore, we have $n\leq\rgr_{R_\fp}(M_\fp)\leq\gd_{R_\fp}(M_\fp)<\depth R_\fp$.
As $\lambda M$ is a first syzygy, $\Ext^1_R(\Tr\lambda M,R)=0$ (cf. \cite[1.4.20 and 1.4.21]{BH}).
Note that $M\cong\lambda^2M$ and so $\Ext^i_R(\Tr\lambda M,R)\cong\Ext^{i-1}_R(\lambda(\lambda M),R)\cong
\Ext^{i-1}_R(M,R)=0$ for $2\leq i\leq n$, and so, by Theorem \ref{AS}, we have
 $\depth_{R_\fp}((\lambda M)_\fp)\geq\min\{n,\depth R_\fp\}$, and so $n\leq\inf\{\depth_{R_\fp}((\lambda M)_\fp)\mid  \fp\in\Spec R,
 \gd_{R_\fp}(M_\fp)\neq0\}.$

 On the other hand, by the Lemma \ref{l3}, if $\fp\in\Ass_R(\Ext^n_R(M,R))$ then $\gd_{R_\fp}(M_\fp)\neq0$ and
 $\depth_{R_\fp}((\lambda M)_\fp)=\rgr(M)$ and so the assertion holds.
 \end{proof}

For a ring $R$, set $X^i(R)=\{\fp\in\Spec R\mid\depth R_\fp\leq i\}$.
We have the following characterization of horizontally linked module of \emph{$G$}-dimension zero.
\begin{prop}\label{Gd}
Let $M$ be a horizontally linked $R$--module of finite \emph{G}-dimension. Then \emph{$\gd_R(M)=0$ if and only if
$\depth_{R_\fp}(M_\fp)+\depth_{R_\fp}((\lambda M)_\fp)>\depth R_\fp$ for all $\fp\in\Spec R\setminus X^0(R)$}.
\end{prop}
\begin{proof}
 Set $n=\gd_R(M)$. Let $\depth_{R_\fp}(M_\fp)+\depth_{R_\fp}((\lambda M)_\fp)>\depth R_\fp$ for all $\fp\in\Spec R\setminus X^0(R)$.
 Assume contrary,
 $n>0$. By Proposition \ref{p}, there exists $\fp\in\Spec R$ such that $\gd_{R_\fp}(M_\fp)\neq0$ and $\rgr(M)=\depth_{R_\fp}
 ((\lambda M)_\fp)$.
 Therefore we have $\depth_{R_\fp}(M_\fp)+\rgr(M)>\depth R_\fp$ and so, by Theorem \ref{Gor}, $\rgr(M)>\gd_{R_\fp}(M_\fp)\geq\rgr(M_\fp)$
 which is a
 contradiction. Therefore $\gd_R(M)=0$.

 Let $\gd_R(M)=0$. Assume contrarily that $\depth_{R_\fp}(M_\fp)+\depth_{R_\fp}((\lambda M)_\fp)\leq\depth R_\fp$ for some
 $\fp\in\Spec R\setminus X^0(R)$. It follows that $\fp\in\Supp_R(M)$. Now \cite[Lemma 1.3.1]{C} implies that $\gd_{R_\fp}(M_\fp)=0$.
 Thus we obtain $\depth_{R_\fp}((\lambda M)_\fp)=0$. As $\lambda M$ is a syzygy, $\depth R_\fp=0$ which is a contradiction.
\end{proof}
Recall that a module $M$ is said to be \emph{horizontally self-linked} if $M\cong\lambda M$ (see \cite[Definition 7]{MS}).
\begin{cor}
Let $M$ be a horizontally self-linked $R$--module of finite \emph{G}-dimension. Then \emph{$\gd_R(M)=0$} if and only if
\emph{$\depth_{R_\fp}(M_\fp) >\frac{1}{2}(\depth R_\fp)$} for all \emph{$\fp\in\Spec R\setminus X^0(R)$}.
\end{cor}
\begin{proof}
As $M\cong\lambda M$, the claim is obvious by Proposition \ref{Gd}.
\end{proof}
For a subset $X$ of $\Spec R$, we say that $M$ is of G-dimension zero on $X$, if $\gd_{R_\fp}(M_\fp)=0$ for all $\fp$
in $X$.
\begin{prop}
Let $M$ be a horizontally linked $R$--module of finite and positive G-dimension. Set
\emph{$t_M=\rgr(M)+\rgr(\lambda M)$}, then $M$ is of G-dimension zero on $X^{t_M-1}(R).$
\end{prop}
\begin{proof}
We prove the claim by induction on $n=\rgr(M)$.
Note that, by Corollary \ref{A}, we have $\depth_{R_\fp}(M_\fp)\geq \min\{\depth R_\fp,\rgr(\lambda M)\}$ for all $\fp\in\Spec R$.
If $n=1$ and $\fp\in X^{t_M-1}(R)$, then it follows that $\gd_{R_\fp}(M_\fp)=\depth R_\fp-\depth_{R_\fp}(M_\fp)=0$.

Assume that $n>1$. It is obvious that $\rgr(M)=\rgr(\Omega M)+1$. Applying induction hypothesis on $\Omega M$, implies that $\Omega M$
is of
G-dimension zero on $X^{t_{\Omega M}-1}(R)$.
 Now consider the exact sequence (\ref{e}).
Note that $\lambda^2\mathcal{T}_iM$ is stably isomorphic to
$\Omega\mathcal{T}_{i+1}M$. Since $\Ext^1_R(M,R)=0$, it follows that
$\underline{\lambda
M}\cong\underline{\Omega^2\mathcal{T}_2(M)}\cong\underline{\Omega(\lambda(\Omega
M))}$ and so $$\Ext^i_R(\lambda M,R) \cong
\Ext^{i+1}_R(\lambda(\Omega M),R)\  \text{for all}\  i>0.$$ Also we
have $\Ext^1_R(\lambda(\Omega M),R)\cong\Ext^2_R(\Tr\Omega M,R)$. As
$M$ is horizontally linked, it is a first syzygy and so $\Omega M$
is a second syzygy. Therefore, by Theorem \ref{AS},
$\Ext^2_R(\Tr\Omega M,R)=0$ and hence $\rgr(\lambda
M)=\rgr(\lambda(\Omega M))-1$ and so $t_{\Omega M}=t_M$. Now
consider the exact sequence $0\rightarrow \Omega M\rightarrow
P\rightarrow M\rightarrow 0$, where $P$ is a projective module.
Since $\Ext^1_R(M,R)=0$, it follows that $M$ is of G-dimension zero
on $X^{t_M-1}(R)$.
\end{proof}
In Theorem \ref{AS}, Auslander and Bridger give the connection of the reduced grade of $\Tr M$ with the property $M$ being a syzygy.
The following result shows the connection of the reduced grade of $M$ with the property that $\lambda M$ being a syzygy.
\begin{prop}\label{c}
Let $M$ be a horizontally linked $R$--module of finite $G$-dimension. Let $k$ be a positive integer. Then the following statements are
equivalent.
\begin{itemize}
           \item[(i)] \emph{$\rgr(M)\geq k$}.
           \item[(ii)] $\lambda M$ is a $k$th syzygy.
           \item[(iii)]Every $R_\fp$-regular sequence of length at most $k$ is also $(\lambda M)_\fp$-regular sequence for each
           \emph{ $\fp\in\Spec R$}.
            \item[(iv)] $\lambda M$ satisfies $\widetilde{S}_k$.
            \end{itemize}
\end{prop}
\begin{proof}
Since $\lambda M$ is a first syzygy, it is torsionless and so $\Ext^1_R(\Tr\lambda M,R)=0$ by (\ref{m}).
Since $M$ is horizontally linked, $M\cong\lambda^2 M$, and so $\Ext^i_R(\Tr\lambda M,R)=0$ for $1\leq i\leq\rgr(M)$, that is
 $\rgr(\Tr\lambda M)>\rgr(M)$.
Replacing $M$ by $\lambda M$ in Theorem \ref{AS}, implies that
(i)$\Rightarrow$(ii)$\Rightarrow$(iii)$\Rightarrow$(iv).

(iv)$\Rightarrow$(i) If $\gd_R(M)=0$ then we have nothing to prove.
Assume that $\gd_R(M)>0$. Set $n=\rgr(M)$ and suppose that
$\fp\in\Ass_R(\Ext^n_R(M,R))$.
 By Lemma \ref{l3}, we have $n=\depth_{R_\fp}((\lambda M)_\fp)$. Also we have $n=\rgr_{R_\fp}(M_\fp)\leq\gd_{R_\fp}(M_\fp)<\depth R_\fp$.
  As $\lambda M$ satisfies $\widetilde{S}_k$, it follows $n=\depth_{R_\fp}((\lambda M)_\fp)\geq\min\{k,\depth R_\fp\}$ and so
   $\depth R_\fp>k$ and so $\rgr(M)\geq k$.
\end{proof}
Let $(R,\fm)$ be a local ring. Recall from \cite{EGS} that $\syz(M)$
denotes the largest $n$ for which $M$ can be an $n$th syzygy in a
minimal free resolution of an $R$--module $N$. Note that
$\syz(M)=\infty$, whenever $\gd_R(M)=0$. If $M$ is a horizontally
linked of finite and positive $G$-dimension then Theorem \ref{AS}
implies that $\syz(M)=\rgr(\lambda M)$. In this case, it is obvious
that $\depth_R(M)\geq\rgr(\lambda M)$. The following theorem shows
that under certain condition the equality holds. In general, it is
well-known that
$\depth_R(M)\leq\depth_{R_{\fp}}(M_{\fp})+\h(\frac{\fm}{\fp})$, for
$\fp\in\Spec R$. In the following result, it is shown that if
$\depth$ of $M$ is small enough then it is equal to the reduced
garde of $\lambda M$. The authors are grateful to J. R. Strooker and
A.-M. Simon for quoting a flaw in the pervious version of this
Theorem.

For an $R$--module $M$, set
$$X(M)=\{\fp\in\Spec(R)\mid\gd_{R_{\fp}}(M_{\fp})\neq0\}.$$
\begin{thm}\label{t}
Let $(R,\fm)$ be a local ring and let $M$ be an $R$--module of finite and positive $G$-dimension. If $M$ is horizontally linked then the following conditions are equivalent.
\begin{itemize}
          \item[(i)]\emph{ $\depth_R(M)=\syz(M)=\rgr(\lambda M)$;}
          \item[(ii)]\emph{$\fm\in\Ass_R(\Ext^{\tiny{\rgr(\lambda M)}}_R(\lambda M,R))$;}
          \item[(iii)]\emph{$\depth_R(M)\leq\depth_{R_{\fp}}(M_{\fp})$, for each $\fp\in X(M)$}.
\end{itemize}
\end{thm}
\begin{proof}
Set $t=\rgr(\lambda M)$. As $\gd_R(M)>0$, by Corollary \ref{A}, $t<\depth R$.
In the following we use the exact sequence (\ref{e}):
\begin{equation}\label{!!}
0\longrightarrow\Ext^t_R(\lambda M,R)\longrightarrow\mathcal{T}_t\lambda M\longrightarrow\lambda^2\mathcal{T}_t\lambda M\longrightarrow0
\end{equation}
As $\Ext^i_R(\lambda M,R)=0$ for all $i$, $1\leq i<t$, $\underline{\mathcal{T}_i(\lambda M)}\cong\underline{\Omega\mathcal{T}_{i+1}(\lambda M)}$ for all $i$, $1\leq i<t$, and so
\begin{equation}\label{!}
\underline{\lambda^2 M}\cong\underline{\Omega^t\mathcal{T}_t(\lambda M)}.
\end{equation}

(i)$\Rightarrow$(ii) As $M$ is horizontally linked and $t=\depth_R(M)<\depth R$, it is easy to see that $\depth_R(\mathcal{T}_t(\lambda M))=0$ by the (\ref{!}). From the exact sequence (\ref{!!}), as $\depth_R(\lambda^2\mathcal{T}_t(\lambda M))>0$, we find that $\depth_R(\Ext^t_R(\lambda M,R))=0$ and so $\fm\in\Ass_R(\Ext^t_R(\lambda M,R))$.

(ii)$\Rightarrow$(i) By the exact sequence (\ref{!!}), it is obvious that $\depth_R(\mathcal{T}_t(\lambda M))=0$. As $M$ is horizontally linked and $t<\depth R$, by (\ref{!}), $\depth_R(M)=t$.

(i)$\Rightarrow$(iii) By Corollary \ref{A}, $M$ satisfies $\widetilde{S}_t$. Hence if $\fp\in X(M)$ then
$\depth_{R_{\fp}}(M_{\fp})\geq t=\depth_R(M)$.

(iii)$\Rightarrow$(i) As $\gd_R(M)>0$, $\depth_R(M)\geq t$, by Corollary \ref{A}. Let $\fp\in\Ass_R(\Ext^{t}_R(\lambda M,R))$. Therefore, $\fp R_{\fp}\in\Ass_{R_{\fp}}(\Ext^t_{R_{\fp}}(\lambda_{R_{\fp}}M_{\fp},R_{\fp}))$ and $\rgr_{R_{\fp}}(\lambda_{R_{\fp}} M_{\fp})=t$.
By the exact sequence (\ref{!!}), we have $\depth_{R_{\fp}}(\mathcal{T}_t(\lambda_{R_{\fp}}M_{\fp}))=0$. If $\gd_{R_{\fp}}(M_{\fp})=0$ then $\gd_{R_{\fp}}(\lambda_{R_{\fp}}M_{\fp})=0$,
which is a contradiction, because $\rgr_{R_{\fp}}(\lambda_{R_{\fp}}M_{\fp})=t<\infty$. Hence $\gd_{R_{\fp}}(M_{\fp})\neq0$. By Corollary \ref{A}, $M$ satisfies $\widetilde{S}_t$ and so it is easy to see that $t<\depth R_{\fp}$. Now by the exact sequence (\ref{!!}), we have $\underline{\mathcal{T}_i(\lambda_{R_{\fp}} M_{\fp})}\cong\underline{\Omega_{R_{\fp}}\mathcal{T}_{i+1}(\lambda_{R_{\fp}} M_{\fp})}$ for all $i$, $1\leq i<t$, and so
$\underline{\lambda^2_{R_{\fp}} M_{\fp}}\cong\underline{\Omega^t_{R_{\fp}}\mathcal{T}_t(\lambda_{R_{\fp}} M_{\fp})}$. Therefore, $\depth_{R_{\fp}}(\lambda^2_{R_{\fp}} M_{\fp})=t$. As $\underline{M_{\fp}}\cong\underline{\lambda^2_{R_{\fp}} M_{\fp}}$, $\depth_{R_{\fp}}(M_{\fp})=t$ and so $\depth_R(M)\leq t$. Therefore, $\depth_R(M)=\rgr(\lambda M)$.
\end{proof}

The following lemma will be useful for the rest of the paper.
\begin{lem}\label{p2}
Let $M$ and $N$ be $R$--modules and let $n$ be a positive integer
with \emph{$\rgr(M)\geq n$}. Then
\begin{itemize}
      \item[(i)]{\emph{$\Tor^R_i(\mathcal{T}_nM,N)\cong \left\lbrace
            \begin{array}{c l}
             \Ext^{n-i}_R(M,N),\ \ & \text{ \ \ $1\leq i<n,$}\\
             \Tor^R_{i-n}(\lambda M,N),\ \ & \text{ \ \ $i>n$;}\\
 \end{array}
\right.$}}\\
\\
       \item[(ii)]{\emph{$\Ext^i_R(\mathcal{T}_nM,N)\cong \left\lbrace
            \begin{array}{c l}
             \Tor^R_{n-i}(M,N),\ \ & \text{ \ \ $1\leq i<n,$}\\
             \Ext^{i-n}_R(\lambda M,N),\ \ & \text{ \ \ $i>n$}.\\
 \end{array}
\right.$}}\\
\end{itemize}
\end{lem}
\begin{proof}
(i) Let $P_\bullet: P_n\rightarrow\cdots\rightarrow P_1\rightarrow
 P_0\rightarrow 0$ be a part of a projective resolution of $M$. Set $F_i=P^*_{n-i}$ for $0\leq i\leq n$, so that one has a right
  acyclic projective
 complex
  $F_\bullet=\Hom_R(P_\bullet,R): F_n \rightarrow\cdots\rightarrow F_1\rightarrow
 F_0\rightarrow 0$, where $\underline{\H_0(F_{\bullet})}\cong\underline{\mathcal{T}_nM}$. Now, the tensor evaluation
 morphism
 $F_\bullet\otimes_RN\cong\Hom_R(P_\bullet,N)$ implies
 \[\begin{array}{rl}
\Tor_i^R(\mathcal{T}_nM,N)&\cong\H_i(F_\bullet\otimes N)\\
&\cong\H^{n-i}(\Hom_R(P_\bullet,N))\\
&\cong\Ext^{n-i}_R(M,N)
\end{array}\]
for all $i$, $1\leq i\leq n-1$.

 On the other hand, the fact that $\underline{\lambda^2\mathcal{T}_iM}\cong\underline{\Omega\mathcal{T}_{i+1}M}$ for all $i>0$,
 and the exact sequence (\ref{e}) in the
  Remark \ref{R}, imply
  $\underline{\lambda M}\cong\underline{\Omega^n\mathcal{T}_nM}$. Therefore $\Tor_i^R(\mathcal{T}_nM,N)\cong\Tor_{i-n}^R(\lambda M,N)$
  for
  $i>n$.

 (ii) By the homomorphism evaluation morphism, one has $ P_{\bullet}\otimes_R N\cong\Hom_R(F_\bullet,N)$, and so
\[\begin{array}{rl}
\Ext^i_R(\mathcal{T}_nM,N)&\cong\H^i(\Hom_R(F_\bullet,N))\\
&\cong\H_{n-i}(P_\bullet\otimes N)\\
&\cong\Tor_{n-i}^R(M,N)
\end{array}\]
 for all $i$, $1\leq i\leq n-1$.
  If $i>n$ then, just similar the part (i), we have $\Ext^{i-n}_R(\lambda M,N)\cong\Ext^i_R(\mathcal{T}_nM,N).$
\end{proof}
\begin{prop}
Let $M$ and $N$ be $R$--modules. Then the following statements hold true.
\begin{itemize}
         \item[(i)] If $M$ is horizontally linked module, then \emph{$\Ext^i_R(M,M)\cong\Ext^i_R(\lambda M,\lambda M)$}
             for all $i$, \emph{$1\leq i<\inf\{\rgr(M),\rgr(\lambda M)\}$}. In particular, if \emph{$\gd_R(M)=0$} then \emph{$\Ext^i_R(M,M)
             \cong\Ext^i_R(\lambda M,\lambda M)$} for all $i>0$.
         \item[(ii)] If $M$ and $N$ are horizontally self-linked modules, then
\emph{$\Ext^i_R(M,N)\cong\Ext^i_R(N,M)$} for all $i$, $1\leq
i<$\emph{$\inf\{\rgr(M),\rgr(N)\}$}.
\end{itemize}
\end{prop}
\begin{proof}
(i) Let $m$ be a positive integer such that
$1<m\leq\inf\{\rgr(M),\rgr(\lambda M)\}$. If $1\leq i<m$ then by
part (i) of Lemma \ref{p2}, one has
$\Ext^i_R(M,M)\cong\Tor_{m-i}^R(\mathcal{T}_mM,M)$. Also, by part
(ii) of Lemma \ref{p2}, one has
$\Tor_{m-i}^R(M,\mathcal{T}_mM)\cong\Ext^i_R(\mathcal{T}_mM,\mathcal{T}_mM)$
for $1\leq i<m$. Since $\rgr(M)\geq m$ it is easy to see that
$\rgr(\mathcal{T}_mM)\geq m$ and $\underline{\Tr M}\cong\underline
{\Omega^{m-1}\mathcal{T}_mM}$ by (\ref{e}). Now, by Theorem \ref{MS}
and the fact that  $M$ is horizontally linked, it follows that
$\Ext^m_R(\mathcal{T}_mM,R) \cong\Ext^1_R(\Tr M,R)=0$. If
$\rgr(\mathcal{T}_mM)<\infty$, then it follows that
$\rgr(\mathcal{T}_mM)=\rgr(\lambda M)+m$ which gives
$\rgr(\mathcal{T}_mM)\geq2m$. The exact sequences
$$0\longrightarrow\Omega^{i+1}\mathcal{T}_mM\longrightarrow P_i\longrightarrow\Omega^{i}\mathcal{T}_mM\longrightarrow 0 \text{ for }
 0\leq i<m,$$
where $P_i$ is a projective module for $0\leq i<m$, imply that
$\Ext^i_R(\mathcal{T}_mM,\mathcal{T}_mM)\cong\Ext^{i+m}_R(\mathcal{T}_mM,\lambda M)\cong\Ext^i_R(\lambda M,\lambda M)$ for
$1\leq i< m$.

(ii) Let $m$ be a positive integer such that
$1<m\leq\inf\{\rgr(M),\rgr(N)\}$. If $1\leq i<m$, then one observes,
similar the part (i), that
$\Ext^i_R(M,N)\cong\Ext^i_R(\mathcal{T}_mN,\mathcal{T}_mM)$ for
$1\leq i<m$. Since $M$ is horizontally self-linked module and
$m\leq\rgr(M)$ it follows from (\ref{e}) that
$\underline{M}\cong\underline {\Omega^m\mathcal{T}_mM}$ and
$\rgr(\mathcal{T}_mM)\geq2m$. Similarly we have
$\underline{N}\cong\underline{\Omega^m\mathcal{T}_mN}$ and
$\rgr(\mathcal{T}_mN)\geq2m$. Therefore
$\Ext^i_R(\mathcal{T}_mN,\mathcal{T}_mM)\cong\Ext^{i+m}_R(\mathcal{T}_mN,
M)\cong\Ext^i_R(N,M)$ for $1\leq i< m$ and the result follows.

\end{proof}

\section{reduced G-perfect modules}
Let $M$ be an $R$--module of positive $G$-dimension. The following inequalities are well-known
$$\gr_R(M)\leq\rgr(M)\leq\gd_R(M)\leq\pd_R(M).$$
In the literature, $M$ is called \emph{perfect} (resp.
\emph{G-perfect}) if $\gr_R(M)=\pd_R(M)$ (resp.
$\gr_R(M)=\gd_R(M)$). In this section we are interested in the case
where $\rgr(M)=\gd_R(M)$. So we bring the following definition.
\begin{defn}\label{rgr}
\emph{Let $M$ be an $R$--module of finite G-dimension, we say that $M$ is \emph{reduced G-perfect} if its $G$-dimension
is equal to its reduced grade, i.e.} \emph{$\rgr(M)=\gd_R(M)$}.
\end{defn}
Note that every reduced G-perfect module has a finite and positive $G$-dimension. It is obvious that $\Ext^{\tiny{\rgr_R(M)}}_R
(M,R)$ is the only non-zero module among all $\Ext^i_R(M,R)$ for $i>0$.
Therefore, every reduced $G$-perfect module is an Eilenberg-MacLane module of finite $G$-dimension as mentioned in \cite{Hor}.
These modules are studied in \cite{N1} and \cite{N2}.
\begin{exam}\cite[2.1(1)]{I}.\label{exa}
\emph{Let $M$ be an $R$--module. If n is a positive integer and
$\gr_R(M)\geq n$, then it is easy to see that $\mathcal{T}_nM$ is a
reduced G-perfect $R$--module of projective dimension $n$. In fact,
the functor $\mathcal{T}_n$ gives a duality between the category of
$R$--modules $X$ with $\gr_R(X)\geq n$ and the category of
$R$--modules $Y$ with $\rgr(Y)\geq n$ and $\pd_R(Y)\leq n.$}
\end{exam}
In the following result, the depth of any reduced \emph{$G$}-perfect module is determined.
\begin{thm}\label{d}
Let $(R,\fm)$ be a local Cohen-Macaulay ring of dimension d. If $M$ is reduced \emph{G}-perfect of \emph{G}-dimension $n$, then
\emph{$\depth_R(M)+\depth_R(\lambda M)=d+\depth_R(\Ext^n_R(M,R))$.}
\end{thm}
\begin{proof} We will use the exact sequence (\ref{e}) in the Remark \ref{R}. Note that $\underline{\lambda^2\mathcal{T}_iM}\cong
\underline{\Omega\mathcal{T}_{i+1}M}$ for all $i>0$. As $M$ is a
reduced G-perfect module, it follows that $\underline{\lambda
M}\cong\underline{\Omega^n\mathcal{T}_nM}$.

We prove the claim by induction on $t=\depth_R(\Ext^n_R(M,R))$. Assume that $t=0$ so that $\fm\in\Ass_R(\Ext^n_R(M,R))$ and then
$\depth_R(\mathcal{T}_nM)=0$ by (\ref{e}). As $\underline{\lambda M}\cong\underline{\Omega^n\mathcal{T}_nM}$, we have $\depth_R
(\lambda M)=n$. Now the equality $\depth_R(M)+\depth_R(\lambda M)=d$, follows by the Auslander-Bridger formula.

Let $t>0$. By \cite[Corollary 4.17]{AB}, $\gr_R(\Ext^n_R(M,R))\geq
n$. Note that $$\dim_R(\Ext^n_R(M,R))=d-\gr_R(\Ext^n_R(M,R))$$ over
the Cohen-Macaulay local ring $R$. Therefore we have
\[\begin{array}{rl}
\depth_R(\Ext^n_R(M,R))&\leq\dim_R(\Ext^n_R(M,R))\\
&\leq d-n\\
&=\depth_R(M).
\end{array}\]
Now let $x\in\fm\smallsetminus\underset{\fp\in X}\cup\fp$, where $X=\Ass_R(M)\cup\Ass(R)\cup\Ass_R(\Ext^n_R(M,R))$. Taking
$\overline{R}=R/xR$ and $\overline{M}=M/xM$,  we have $\gd_R(M)=\gd_{\overline{R}}(\overline{M})$ by \cite[Corollary 4.30]{AB}.
 Apply the functor $\Hom_R(M,-)$ on the exact sequence
$0\rightarrow R\overset{x}\rightarrow R\rightarrow \overline{R}\rightarrow0$. As $x$ is a nonzero-divisor on $\Ext^n_R(M,R)$,
we obtain the following exact sequences
\begin{equation}\label{31}
0\longrightarrow M^*\overset{x}\longrightarrow M^*\longrightarrow \Hom_R(M,\overline{R})\longrightarrow 0
\end{equation}
\begin{equation}\label{32}
 0\longrightarrow\Ext^n_R(M,R)\overset{x}\longrightarrow \Ext^n_R(M,R)\longrightarrow \Ext^n_R(M,\overline{R})\longrightarrow 0.
\end{equation}
Moreover, $\Ext^i_R(M,\overline{R})=0$ for $1\leq i<n$.
From (\ref{31}), (\ref{32}) and the standard isomorphisms we get
$\Hom_{\overline{R}}(\overline{M},\overline{R})\cong\overline{\Hom_R(M,R)}$ and
$\Ext^n_{\overline{R}}(\overline{M},\overline{R})\cong\overline{\Ext^n_R(M,R)}$.
Now it is easy to see that $\rgr_{\overline{R}}(\overline{M})=\rgr_R(M)$. As a result, $\overline{M}$ is also a reduced G-perfect
 $\overline{R}$--module of $G$-dimension $n$. As $\depth_{\overline{R}}(\Ext^n_{\overline{R}}(\overline{M},\overline{R}))
 =\depth_{\overline{R}}(\overline{\Ext^n_R(M,R)})=t-1$,
\begin{equation}\label{33}
\depth_{\overline{R}}(\overline{M})+\depth_{\overline{R}}(\lambda_{\overline{R}}\overline{M})=\dim\overline{R}+
\depth_{\overline{R}}(\Ext^n_{\overline{R}}(\overline{M},\overline{R})),
\end{equation}
by induction hypothesis.

Now let $P_1\longrightarrow P_0\longrightarrow M\longrightarrow 0$
 be the minimal projective presentation of $M$ and consider the exact sequence $0\longrightarrow M^*\longrightarrow
 P^*_0\longrightarrow \lambda M\longrightarrow 0$. As $\lambda M$ is a syzygy module and $x$ is a non-zero-divisor on $R$, $x$
  is also a non-zero-divisor on
 $\lambda M$. Thus there is a commutative diagram
with exact rows\\
$$\begin{CD}
&&&&&&&&\\
\ \ &&&& 0 @>>> M^*/{xM^*} @>>>P^*_0/{xP^*_0} @>>> {\lambda M}/{x\lambda M} @>>>0&  \\
&&&&&&  @VV{\cong}V @VV{\cong}V \\
\ \  &&&& 0 @>>>\Hom_{\overline R}(\overline M,\overline R) @>>> \Hom_{\overline R}(\overline P_0,\overline R)
 @>>>\lambda_{\overline{R}}\overline{M} @>>>0&\\
\end{CD}$$\\
which implies that $\lambda M/{x\lambda M} \cong \lambda_{\overline{R}}\overline{M}$. Now, (\ref{33}) implies
$$\depth_R(M)+\depth_R(\lambda_RM)=\dim R+\depth_R(\Ext^n_R(M,R)).$$
\end{proof}
Theorem \ref{d} enables us to prove it for a more general setting.
\begin{thm}\label{t1}
Let $R$ be a Cohen-Macaulay local ring of dimension $d$ and let $M$
be an $R$--module of finite Gorenstein dimension $n$. If
\emph{$$\depth_R(\Ext^n_R(M,R))< \depth_R(\Ext^{n-i}_R(M,R))-i-1 \
\text{\it for all}\ i, 0<i<n,$$}
 then \emph{$\depth_R(M)+\depth_R(\lambda M)=
d+\depth_R(\Ext^n_R(M,R))$}.
\end{thm}
\begin{proof}
First note that if $n=0$ then $\gd_R(M^*)=\gd_R(\lambda M)=0$ and so
the assertion is obvious. We proceed by induction on $n$. If $n=1$
then $\gd_R(M)=\rgr(M)$ that is $M$ is reduced G-perfect. Hence, by
Theorem \ref{d}, the assertion holds. Suppose that $n>1$. We have
$\gd_R(\Omega M)=n-1$ and $\Ext^i_R(M,R)\cong\Ext^{i-1}_R(\Omega
M,R)$ for each $i>1$. By induction hypothesis we have
\begin{equation}\label{3-3}\depth_R(\Omega M)+\depth_R(\lambda\Omega
M)=d+\depth_R(\Ext^n_R(M,R)).\end{equation}
 Therefore
$\depth_R(\lambda\Omega M)+1=n+\depth_ (\Ext^n_R(M,R))<
\depth_R(\Ext^1_R(M,R))$. Now consider the exact sequence
$0\rightarrow\Ext^1_R(M,R)\rightarrow\Tr M\rightarrow\lambda^2\Tr
M\rightarrow0$. As $\underline{\lambda^2\Tr
M}\cong\underline{\lambda\Omega M}$, it easy to see that
$\depth_R(\lambda\Omega M)= \depth_R(\Tr M)=\depth_R(\lambda M)-1$.
Finally, by using the trivial formula $\depth_R(\Omega
M)=\depth_R(M)+1$ and (\ref{3-3}), we have
$\depth_R(M)+\depth_R(\lambda M)= d+\depth_R(\Ext^n_R(M,R))$.
\end{proof}
The following result characterizes a reduced G-perfect module to be
horizontally linked.
\begin{prop}\label{P1}
Let $M$ be a reduced \emph{G}-perfect $R$--module of \emph{G}-dimension $n$, then the following statements hold true.
\begin{itemize}
     \item[(i)] {\emph{$\Ext^i_R(\lambda M,R)\cong\Ext^{n+i}_R(\Ext^n_R(M,R),R)$} for all $i>0$.}
      \item[(ii)]{Assume that $M$ is stable $R$--module. Then $M$ is horizontally linked if and only if
      \emph{$\rgr(M)+\rgr(\lambda M)=\gr_R(\Ext^n_R(M,R)).$}}
\end{itemize}
\end{prop}
\begin{proof}
(i) As $M$ is reduced G-perfect, $\rgr(M)=\gd_R(M)=n$. Consider the
exact sequence (\ref{e}) for $i=n$
\begin{equation}
 0\longrightarrow \Ext^n_R(M,R)\longrightarrow \mathcal{T}_nM\longrightarrow \lambda^2
\mathcal{T}_nM\longrightarrow 0.\label{n}
\end{equation}
As $\gd_R(M)=n$, it follows that $\gd_R(\Omega^nM)=0$ and so
$\gd_R(\mathcal{T}_{n+1}M)=0$ by \cite[Lemma 4.9]{AB}. Therefore
$\gd_R(\Omega\mathcal{T}_{n+1}M)=0$. As
$\underline{\lambda^2\mathcal{T}_nM}\cong\underline{\Omega\mathcal{T}_{n+1}M}$,
we have $\gd_R(\lambda^2\mathcal{T}_nM)=0$.
 Now from the exact sequence (\ref{n}) we obtain the isomorphisms
$\Ext^i_R(\Ext^n_R(M,R),R)\cong\Ext^{i}_R(\mathcal{T}_nM,R)$ for all $i>0$. Note that $\rgr(M)=n$ implies
$\underline{\lambda M}\cong\underline{\Omega^{n}\mathcal{T}_nM}$ and so $\Ext^i_R(\lambda M,R)\cong\Ext^{n+i}_R(\mathcal{T}_nM,R)$
 for all $i>0$. The claim follows.

(ii) Assume that $M$ is stable $R$--module. By Theorem \ref{MS} and Theorem \ref{AS}, $M$ is horizontally linked module if and only if
 $\gr_R(\Ext^n_R(M,R))\geq n+1$. Now by the part (i), it is obvious that $M$ is horizontally linked module if and only if
  $\rgr(M)+\rgr(\lambda M)=\gr_R(\Ext^n_R(M,R))$.
\end{proof}
The following Corollary shows how reduced G-perfect property is
preserved under horizontally linkage.
\begin{cor}
Let $M$ be a horizontally linked $R$--module. Let $n$ and $t$ be two integers, then the following statements are equivalent.
\begin{itemize}
   \item[(i)] $M$ is reduced \emph{G}-perfect of \emph{G}-dimension $n$ and \emph{$\Ext^n_R(M,R)$} is \emph{G}-perfect of
   \emph{G}-dimension $n+t$.
     \item[(ii)]$\lambda M$ is reduced \emph{G}-perfect of \emph{G}-dimension $t$ and \emph{$\Ext^t_R(\lambda M,R)$} is
     \emph{G}-perfect of \emph{G}-dimension $n+t$.
\end{itemize}
\end{cor}
\begin{proof}
 (i)$\Rightarrow$(ii). As we have seen in the proof of Proposition \ref{P1}, $\gd_R(\lambda^2\mathcal{T}_nM)=0$ and also
 $\underline{\lambda M}\cong\underline{\Omega^n\mathcal{T}_nM}$. Therefore, from the exact sequence (\ref{n}), we conclude
 that $\gd_R(\lambda M)<\infty$ if and only if
$\gd_R(\Ext^n_R(M,R))<\infty$. Also, from the part (i) of Proposition \ref{P1}, if these dimensions are finite then we have
 $\gd_R(\lambda M)+n=\gd_R(\Ext^n_R(M,R))$. Hence, by part (ii)
of Proposition \ref{P1}, it is obvious that $\lambda M$ is reduced
G-perfect of G-dimension $t$ if and only if $\Ext^n_R(M,R)$ is G-perfect of G-dimension $n+t$.

As $M$ is horizontally linked, $M\cong\lambda^2M$, and $M$ is reduced G-perfect of G-dimension $n$, by replacing $M$ by
$\lambda M$ in the above argument one has $\Ext^t_R(\lambda M,R)$ is G-perfect of G-dimension $n+t$.

(ii)$\Rightarrow$(i). It can be proved by the same argument.

\end{proof}
\begin{prop}
Let $M$ be a horizontally linked $R$--module of \emph{G}-dimension $n$ with $0<n<\infty$. Then $M$ is reduced \emph{G}-perfect
if and only if $\lambda M$ satisfies $\widetilde{S}_n$.
\end{prop}
\begin{proof}
As $M$ has a positive and finite Gorenstein dimension $n$, $\rgr(M)\leq\gd_R(M)$. On the other hand, by Proposition \ref{c},
$\lambda M$ satisfies $\widetilde{S}_n$ if and only if $\rgr(M)\geq n$. Therefore the assertion is obvious.
\end{proof}

Let $M$ be a perfect $R$--module of projective dimension $n$ and let
$N$ be an $R$--module. It is well-known that
$\Tor^R_i(\Ext^n_R(M,R),N)\cong\Ext^{n-i}_R(M,N)$ for all $i$,
$0\leq i<n$.

To find an answer for the case $i> n$, we consider $M$ to be reduced
$G$-perfect and put some conditions on $N$. More precisely
\begin{prop}\label{P5}
Let $M$ be a reduced \emph{G}-perfect $R$--module of \emph{G}-dimension $n$ and let $N$ be an $R$--module of finite flat or injective
dimension. Then we have the following isomorphisms.
\begin{itemize}
     \item[(i)]{\emph{$\Tor^R_i(\Ext^n_R(M,R),N)\cong \left\lbrace
               \begin{array}{c l}
               \Ext^{n-i}_R(M,N)\ \ & \text{ \ \ $0\leq i<n$}\\
               \Tor^R_{i-n}(\lambda M,N)\ \ & \text{ \ \ $i>n$,}\\
 \end{array}
\right.$}}\\
\\
\item[(ii)]{\emph{$\Ext^i_R(\Ext^n_R(M,R),N)\cong \left\lbrace
            \begin{array}{c l}
             \Tor^R_{n-i}(M,N)\ \ & \text{ \ \ $0\leq i<n$,}\\
             \Ext^{i-n}_R(\lambda M,N)\ \ & \text{ \ \ $i>n$}.\\
 \end{array}
\right.$}}
\end{itemize}
\end{prop}
\begin{proof}
As $\gd_R(M)=n$, one has $\gd_R(\Omega^nM)=0$ and so $\gd_R(\mathcal{T}_{n+1}M)=0$.
Now, by \cite[Theorem 2.10 and Theorem 2.11]{Y},
we have $\Ext^i_R(\mathcal{T}_{n+1}M,N)=0=\Tor_i^R(\mathcal{T}_{n+1}M,N)$ for all $i>0$.
Therefore, from the exact sequences
$$0\rightarrow\Ext^1_R(\mathcal{T}_{n+1}M,N)\rightarrow\Tor_n^R(M,N)\rightarrow\Hom_R(\Ext^n_R(M,R),N)
\rightarrow\Ext^2_R(\mathcal{T}_{n+1}M,N)$$
and
$$\Tor_2^R(\mathcal{T}_{n+1}M,N)\rightarrow\Ext^n_R(M,R)\otimes_RN\rightarrow\Ext^n_R(M,N)\rightarrow
\Tor_1^R(\mathcal{T}_{n+1}M,N)\rightarrow0,$$ of \cite[Theorem
2.8]{AB}, we get $\Tor_n^R(M,N)\cong\Hom_R(\Ext^n_R(M,R),N)$ and
$\Ext^n_R(M,R)\otimes_RN\cong\Ext^n_R(M,N)$, which prove the claims
(i) and (ii) for $i=0$. For $i>0$, we consider the exact sequence
$0\rightarrow \Ext^n_R(M,R)\rightarrow \mathcal{T}_nM\rightarrow
\lambda^2\mathcal{T}_nM\rightarrow 0.$ Note that
$\underline{\lambda^2\mathcal{T}_nM}\cong\underline{\Omega\mathcal{T}_{n+1}M}$,
and so $\gd_R(\lambda^2\mathcal{T}_nM)=0$.
 Therefore $\Ext^i_R(\Ext^n_R(M,R),N)\cong\Ext^i_R(\mathcal{T}_nM,N)$ and also
$\Tor_i^R(\Ext^n_R(M,R),N)\cong\Tor_i^R(\mathcal{T}_nM,N)$ for all $i>0$.
Now the assertion is obvious by Lemma \ref{p2}.
 \end{proof}

\section{linkage of modules and local cohomology}

Let $\fa$ and $\fb$ be ideals in a Gorenstein local ring $R$ which are linked by a Gorenstein ideal $\fc$. In \cite[Theorem 4.1]{Sc},
 Schenzel proves
that the Serre condition $(S_r)$ for $R/{\fa}$ is equivalent to the vanishing
of the local cohomology groups $\H^i_{\fm}(R/{\fb})= 0$ for all $i$, $\dim(R/{\fb})-r<i<\dim(R/{\fb})$. Here we extend this result
for any
horizontally linked module of finite \emph{G}-dimension over a more general ground ring, i.e. over a Cohen-Macaulay local ring.
First we bring the following lemma which is clear if the ground ring is Gorenstein by using the Local Duality Theorem.
\begin{lem}\label{l1}
Let $R$ be a Cohen-Macaulay local ring of dimension $d$ and let $M$ be an $R$--module of dimension $d$ which is not maximal
Cohen-Macaulay.
If \emph{$\gd_R(\lambda M)<\infty$ then $$\sup\{i\mid\H^i_{\fm}(M)\neq0, i\neq d\}=d-\rgr(M).$$}
\end{lem}
\begin{proof}
We may assume that $R$ is complete with canonical module $\omega_R$. Set
$t=\sup\{i\mid\H^i_{\fm}(M)\neq0, i\neq d\}$ and $n=\rgr(M)$. Note that $n\leq d$. Otherwise, $n>d$ implies
$\underline{\Tr M}\cong\underline{\Omega^d\mathcal{T}_{d+1}M}$ by (\ref{e}) and so
$\gd_R(\Tr M)=0$ by Theorem \ref{Gor} (ii). Therefore $\gd_R(M)=0$ which is a contradiction.
Note that, by Lemma \ref{p2} (i), we have
\begin{equation}\label{41}
\Ext^{n-i}_R(M,\omega_R)\cong\Tor_i^R(\mathcal{T}_nM,\omega_R)  \text{ for } 1\leq i\leq n-1.
\end{equation}
Since $\underline{\Omega^n\mathcal{T}_nM}\cong\underline{\lambda
M}$, it follows $\gd_R(\mathcal{T}_nM)<\infty$. By \cite[Proposition
2.5]{F}, we have $\Tor_i^R(\mathcal{T}_nM,\omega_R)=0$ for all
$i>0$. Now, the Local Duality Theorem and (\ref{41}) imply that
$t\leq d-n$. If $t=d-1$ then we have nothing to prove. Assume
that $t<d-1$. Let $P_\bullet: \cdots\rightarrow P_1\rightarrow
P_0\rightarrow 0$ be a minimal projective resolution of $M$. We
have the commutative diagram
$$\begin{CD}
&&&&&&&&\\
  \ \ &&&&  P^*_0\otimes_R\omega_R @>>>P^*_1\otimes_R\omega_R @>>> \Tr M\otimes_R\omega_R @>>>0&  \\
  &&&& @VV{\cong}V @VV{\cong}V \\
  \ \  &&&& P^{\triangledown}_0 @>f>> P^{\triangledown}_1
   @>>>\coker(f) @>>>0&\\
\end{CD}$$\\
with exact rows, where $P^{\triangledown}\cong\Hom_R(P,\omega_R)$. Therefore $\Tr M\otimes_R\omega_R\cong \coker(f)$.
On the other hand, by the Local Duality Theorem, we have $\Ext^i_R(M,\omega_R)=0$ for all $i$, $1\leq i<d-t$, from which we obtain
the exact sequence
$0\rightarrow \coker(f)\rightarrow P^{\triangledown}_2\rightarrow\cdots\rightarrow P^{\triangledown}_{d-t}$. Now it is
straightforward to see that
$\coker(f)$, and so $\Tr M\otimes_R\omega_R$ satisfies $\widetilde{S}_{d-t-1}$.

Since $\gd_R(\Tr M)<\infty$, it follows, by \cite[Theorem 2.13]{Y} and \cite[Proposition 2.5]{F}, that
$\Tr M$ satisfies $\widetilde{S}_{d-t-1}$. As $\underline{\Tr\Tr M}\cong\underline{M}$,
$\Ext^i_R(M,R)=0$ for $1\leq i\leq d-t-1$ by Theorem \ref{AS}. Hence $\rgr(M)=d-t$.
\end{proof}
Now we are able to generalize \cite[Theorem 4.1]{Sc} for modules of finite Gorenstein dimension.
\begin{thm}
Let $R$ be a Cohen-Macaulay local ring of dimension d, and let $M$ be a horizontally linked $R$--module of finite G-dimension.
Let $k$ be a non-negative integer. Then
$M$ satisfies the Serre condition $(S_k)$ if and only if \emph{$\H^i_{\fm}(\lambda M)=0$} for all $i$, $d-k+1\leq i\leq d-1$.
\end{thm}
\begin{proof}
As $\Ass_R(M)\subseteqq\Ass R$, it is easy to see that $M$ satisfies $(S_k)$ if and only if $M$ satisfies the condition $\widetilde{S}_k$ which is also equivalent to say $\rgr(\lambda M)\geq k$, by
Corollary \ref{A}. Now the claim follows from the proof of Lemma \ref{41}, by replacing $M$ by $\lambda M$.
\end{proof}
\begin{prop}$($Compare with \cite[Corollary 3.2]{DGHS}$)$
Let $(R,\fm)$ be a local ring with \emph{$\depth R>0$} and let $M$ be a stable $R$--module. Assume that $M$ is of
\emph{G}-dimension zero on the
punctured spectrum of $R$. Then the following statements hold true.
\begin{itemize}
        \item[(i)]{$M$ is horizontally linked if and only if $\Gamma_{\fm}(M)=0$.}
         \item[(ii)] If $M$ is horizontally linked module of finite and positive G-dimension, then

        \emph{$\rgr(M)=\depth_R(\lambda M)\ \text{and}\ \rgr(\lambda M)=\depth_R(M)$}.
\end{itemize}
\end{prop}
\begin{proof}
(i) As $M$ is locally G-dimension zero on the punctured spectrum and $\underline{(\Tr M)_\fp}\cong\underline{\Tr_{R_\fp}M_\fp}$ for
all $\fp\in\Spec R$, it follows that $\Tr M$ is also of locally G-dimension zero on the punctured spectrum and so $\Supp_R(\Ext^1_R(\Tr
M,R))\subseteqq\{\fm\}$. Consider the exact sequence (\ref{e1}),
$0\rightarrow\Ext^1_R(\Tr M,R)\rightarrow M\rightarrow \lambda^2 M\rightarrow 0$.
As $\lambda^2M$ is a syzygy module and $\fm\not\in\Ass R$, we get
$\Ext^1_R(\Tr M,R)\cong\Gamma_{\fm}(\Ext^1_R(\Tr M,R))\cong\Gamma_{\fm}(M)$. Now the assertion is obvious by Theorem \ref{MS}.

(ii) Set $\rgr(M)=t$. As $t>0$ and $M$ is of G-dimension zero on the punctured spectrum, it follows that $\Ass_R(\Ext^t_R(M,R))=\{\fm\}$.
Now Lemma \ref{l3} implies that $\depth_R(\lambda M)=\rgr(M)$.
\end{proof}
In \cite[Corollary 3.4]{DGHS}, it is shown that $\Ext^i_R(M,R)\cong\H^i_{\fm}(\lambda M)$ for all $i$, $1\leq i< \dim R$ whenever $R$ is
 Cohen-Macaulay with canonical module $\omega_R$, $\Tor_i^R(M,\omega_R)=0$ for all $i>0$ and $\Ext^i_R(M,R)$ is of finite length
 for all $i$, $1\leq i< \dim R$. In the following, we conclude the same result without assuming that $R$ is Cohen-Macaulay and without
 condition on torsion modules of the canonical module.
\begin{thm}
Let $R$ be a local ring with \emph{$\depth R\geq2$} and let $M$ be an $R$--module. Assume that $n$ is an integer such that
\emph{$1< n\leq\depth R$} and that \emph{$\Ext^i_R(M,R)$} is of finite length
for all $i$, $1\leq i<n$.
Then \emph{$\Ext^i_R(M,R)\cong\H^i_{\fm}(\lambda M)$} for all $i$, $1\leq i<n$.
\end{thm}
\begin{proof}
Consider the exact sequence (\ref{e}),
$0\longrightarrow\Ext^i_R(M,R)\longrightarrow\mathcal{T}_iM\longrightarrow\lambda^2\mathcal{T}_{i}M\longrightarrow 0$.
Note that
$\underline{\lambda^2\mathcal{T}_{i}M}\cong\underline{\Omega\mathcal{T}_{i+1}M}$. Since $\Ext^i_R(M,R)$ is of finite length
for $1\leq i<n$,
applying the functor $\Gamma_{\fm}(-)$ on this exact sequence, we get
\begin{equation}\label{42}
\H^j_{\fm}(\mathcal{T}_{i-1}M)\cong\H^j_{\fm}(\Omega\mathcal{T}_{i}M) \text{ for all}\ i\ \ \text{and}\ j, \text{with}\ 1\leq j<n,\ 1<i\leq n
\end{equation}
 and also
\begin{equation}\label{43}
\Ext^i_R(M,R)=\Gamma_{\fm}(\Ext^i_R(M,R))\cong\Gamma_{\fm}
(\mathcal{T}_iM) \text{ for all}\ i, 1\leq i<n.
\end{equation}
On the other hand,
\begin{equation}\label{44}
\H^j_{\fm}(\mathcal{T}_iM)\cong\H^{j+1}_{\fm}(\Omega\mathcal{T}_iM)
\text{ for all}\ i\ \text{and}\ j,  0\leq j<n-1,\  1\leq i<n.
\end{equation}
Now by using (\ref{43}), (\ref{44}) and (\ref{42}) we obtain the result.
\end{proof}
\section{Semidualizing modules and evenly linked modules}
Let $(R,\fm)$ be a Gorenstein local ring, $\fc_1$ and $\fc_2$ Gorenstein ideals. Let $M_1,M$ and $M_2$ be $R$--modules
such that $M_1$ is linked to $M$ by $\fc_1$ and $M$ is linked to $M_2$ by $\fc_2$. In \cite[Proposition 15 and Proposition 16]{MS},
Martsinkovsky and Strooker prove that $\gd_R(M_1)=\gd_R(M_2)$ and also
\begin{equation}\label{51}
\Ext^i_{R/\fc_1}(M_1,R/\fc_1)\cong\Ext^i_{R/\fc_2}(M_2,R/\fc_2) \text{ for all } i>0.
\end{equation}
In this section we prove (\ref{51}), without assuming $R$ is Gorenstein, but we assume some conditions on the modules $M_1$,
$M$, $M_2$ and on ideals $c_1$, $c_2$.
We also show that the condition $\widetilde{S}_k$(see Definition \ref{S}) is preserved under evenly linkage.\\

Throughout this sectsion $R$ is a local ring, $K$ and $M$ are $R$--modules. Denote $M^\dagger=\Hom_R(M,K)$.
The module $M$ is called \emph{$K$-reflexive} if the canonical map $M\rightarrow M^{\dagger\dagger}$
is bijective.

The Gorenstein dimension was extended to $\gk$dimension by Foxby in \cite{F} and by Golod in \cite{G}.
\begin{defn}
\emph{The module $M$ is said to have \emph{$\gk$dimension zero} if}
\begin{itemize}
            \item[(i)]{\emph{$M$ is $K$-reflexive;}}
             \item[(ii)]{\emph{$\Ext^i_R(M,K)=0$}, for all $i>0$;}
              \item[(iii)]{\emph{$\Ext^i_R(M^{\dagger},K)=0$}, for all $i>0$.}
\end{itemize}
\end{defn}
A $\gk$resolution of a finite $R$--module $M$ is a right acyclic complex of modules of  $\gk$dimensions zero whose $0$th
homology module is $M$. The module $M$ is said to have finite $\gk$dimension, denoted by $\gkd_R(M)$, if
it has a $\gk$resolution of finite length.\\
Recall definition of semidualizing modules and some of its properties.
\begin{defn}
\emph{An $R$--module $K$ is called a \emph{semidualizing} module (suitable), if}
\begin{itemize}
    \item[(i)]{the homothety morphism \emph{$R\rightarrow\Hom_R(K,K)$} is an isomorphism;}
    \item[(ii)]{\emph{$\Ext^i_R(K,K)=0$} for all $i>0$.}
\end{itemize}
\end{defn}
Semidualizing modules are studied in \cite{F} and \cite{G}.
It is obvious that $R$ itself is a semidualizing $R$--module. Also if $R$ is Cohen-Macaulay then its canonical module (if exists)
is a semidualizing module.

Recall that the $\gk$dimension of a module $M$ can be expressed as follows, if it is finite.
\begin{thm}\emph{(\cite[4,8]{G})}\label{G3}
Let $K$ be a semidualizing $R$--module. For an $R$--module $M$ of finite $\gk$dimension the following statements
hold true.
\begin{itemize}
       \item[(i)]{\emph{$\gkd_R(M)=\sup\{i\mid\Ext^i_R(M,K)\neq0, i\geq0\}$,}}
        \item[(ii)]{\emph{$\gkd_R(M)=\depth R-\depth_R(M)$.}}
\end{itemize}
\end{thm}
We recall the following definitions from \cite{G}.
\begin{defn}
\emph{An $R$--module $M$ is called \emph{$\gk$perfect} if $\gr_R(M)=\gkd_R(M)$. An ideal $I$ is called $\gk$perfect if $R/I$ is
$\gk$perfect as $R$--module. An $R$--module $M$ is called \emph{$\gk$Gorenstein} if $M$ is $\gk$perfect
and $\Ext^n_R(M,K)$ is cyclic, where $n=\gkd_R(M)$. An ideal $I$ is called $\gk$Gorenstein if $R/I$ is $\gk$Gorenstein as $R$--module.}
\end{defn}
Note that if $K$ is a semidualizing $R$--module and $I$ is a $\gk$Gorenstein ideal of $\gk$dimension $n$, then
$\Ext^n_R(R/I,K)\cong R/I$(see \cite[10]{G}).

We frequently use the following results of Golod.
\begin{lem}\label{G1} \cite[Corollary]{G}
Let $I$ be an ideal of $R$. Assume that $K$ is an $R$--module and that $n$ is a fixed integer. If \emph{$\Ext^j_R(R/I,K)=0$} for
all $j\neq n$ then there is an isomorphism of functors \emph{$\Ext^i_{R/I}(-,\Ext^n_R(R/I,K))
\cong\Ext^{n+i}_R(-,K)$} on the category of $R/I$--modules for all $i\geq0$.
\end{lem}

\begin{thm}\label{G2}
\cite[Proposition 5]{G}. Let $I$ be a \emph{$\gk$}perfect ideal, and let K be a semidualizing
$R$--module. Set $C=$\emph{$\Ext^{\tiny{\gr(I)}}_R(R/I,K)$}. Then the following statements hold true.
\begin{itemize}
        \item[(i)]{$C$ is a semidualizing $R/I$--module.}
         \item[(ii)]{If $M$ is a $R/I$--module then \emph{$\gkd_R(M)<\infty$} if and only if \emph{$\gc_{R/I}(M)<\infty$}, and
         if these dimensions are finite then
\emph{$\gkd_R(M)=\gr(I)+\gc_{R/I}(M)$.}}
\end{itemize}
\end{thm}
Let $I$ be a perfect ideal (i.e. $R/I$ is
perfect as an $R$--module). Assume that $\gr(I)=n$ and that $\Ext^n_R(R/I,R)$ is cyclic. Let $M$ be a perfect $R$--module
of grade $n$ such that $I\subseteq\Ann_R(M)$. In \cite[Theorem 15]{MS}, Martsinkovsky and Strooker prove that if $M$ is stable
as an $R/I$--module then
$M$ is linked by $I$ and also $\lambda_{R/I}M$ is perfect of grade $n$.
In the following, we give a generalization of this Theorem.
\begin{prop}\label{P3}
Let $K$ be a semidualizing $R$--module. Assume that $I$ is a \emph{$\gk$}Gorenstein ideal of grade $n$ and that $M$ is a
\emph{$\gk$}perfect
$R$--module of grade $n$ such that \emph{$I\subseteq\Ann_R(M)$}. Then $\lambda_{R/I}M$ is a  \emph{$\gk$}perfect $R$--module of grade n.
 Moreover, if $M$ is stable as $R/I$--module then it is horizontally linked as $R/I$--module.
\end{prop}
\begin{proof}
Since $I$ is a $\gk$Gorenstein ideal, $\Ext^n_R(R/I,K)\cong R/I$ (see \cite[10]{G}). Now by Theorem \ref{G2}(ii), $\gd_{R/I}(M)=0$ so
that $\gd_{R/I}(\Tr_{R/I}M)=0$ (see \cite[Lemma 4.9]{AB}), and hence $\gd_{R/I}(\lambda_{R/I}M)=0$. Again by applying Theorem \ref{G2}(ii)
for $\lambda M$, we conclude that $\gkd_R(\lambda_{R/I}M)=n$. As $I\subseteq\Ann_R(\lambda_{R/I}M)$, we have
$n\leq\gr_R(\lambda_{R/I}M)\leq\gkd_R(\lambda_{R/I} M)$. Hence $\lambda_{R/I}M$ is a $\gk$perfect $R$--module of grade
$n$. Now, if $M$ is stable as $R/I$--module then it is horizontally linked
by \cite[Theorem 1]{MS}.
\end{proof}
Recall that two $R$--modules $M$ and $N$ are said to be in the same even
linkage class, or evenly linked, if there is a chain of even length of linked modules
that starts with $M$ and ends with $N$.

We first bring the following lemma which is a generalization of \cite[Lemma 14]{MS}.
\begin{lem}\label{l2}
Assume that $I$ is an ideal of $R$ and that $M$ is horizontally linked as $R/I$--module.
Assume that $K$ is a semidualizing $R$--module and that $I$ is $\gk$Gorenstein. Then \emph{$\gr_R(M)=\gr(I)$}.
\end{lem}
\begin{proof}
Set $n=\gr(I)$. As $I\subseteq\Ann_R(M)$, $\gr_R(M)\geq n$. On the other hand, as $I$ is $\gk$Gorenstein,
$\Ext^n_R(M,K)\cong\Hom_{R/I}(M,R/I)$, by the Lemma \ref{G1}. Since $M$ is horizontally linked $R/I$--module, $\lambda_{R/I} M$
 is stable by \cite[Proposition 4]{MS}.
 Therefore $\Hom_{R/I}(M,R/I)\neq0$ by (\ref{2}). Since $K$ is a semidualizing $R$--module it follows that every $R$-regular
 sequence is also $K$-regular sequence (see \cite[1]{G}). Therefore, by \cite[Proposition 1.2.10]{BH}, we have
\[\begin{array}{rl}
\gr_R(M)&=\gr(\Ann_R(M),R)\\
&\leq\gr(\Ann_R(M),K)\\
&=\inf\{i\geq0\mid\Ext^i_R(M,K)\neq0\}\\
&\leq n.
\end{array}\]
\end{proof}
The following Proposition is a generalization of \cite[Proposition 16]{MS}.
\begin{prop}\label{p4}
Let $K$ be a semidualizing $R$--module, $\fc_1$ and $\fc_2$ two $\gk$Gorenstein ideals.
Assume that $M_1$,$M$, and $M_2$ are $R$--modules such that $M_1\underset{\fc_1}{\thicksim}M$ and $M\underset{\fc_2}{\thicksim}M_2$.
Denote the common value of  \emph{$\gr(\fc_1)$} and \emph{$\gr(\fc_2)$} by $n$ (see Lemma \ref{l2}). Then \emph{$\Ext^i_R(M_1,K)\cong
\Ext^i_R(M_2,K)$} for all $i$, $i>n$.
\end{prop}

\begin{proof}
By Lemma \ref{G1}, it is enough to show that $\Ext^i_{R/{\fc_1}}(M_1,R/{\fc_1})\cong\Ext^i_{R/{\fc_2}}(M_2,R/{\fc_2})$ for all $i>0$.
Note that, by (\ref{m}), one has
\[\begin{array}{rl}
\Ext^1_{R/{\fc_1}}(M_1,R/{\fc_1})&\cong\Ext^1_{R/{\fc_1}}(\lambda_{R/{\fc_1}}M,R/{\fc_1})\\
&\cong\Ext^2_{R/{\fc_1}}(\Tr_{R/{\fc_1}}M,R/{\fc_1})\\
&\cong\coker e_M,
\end{array}\]
where $e_M:M\rightarrow\Hom_{R/{\fc_1}}(\Hom_{R/{\fc_1}}(M,R/{\fc_1}),R/{\fc_1})$ is the canonical evaluation map.
Note that $\gr(\fc_1\cap \fc_2,R)=n$ (c.f. \cite[Proposition 1.2.10]{BH}).
Let $\underline{x}\subseteq\fc_1\cap\fc_2$
be an $R$-regular sequence of length $n$. As $K$ is a semidualizing $R$--module, $\underline{x}$ is also a $K$-regular sequences.
Set $\overline{R}=R/{\underline{x}R}$ and $\overline{K}=K/{\underline{x}K}$. By the standard isomorphisms and the fact that $\fc_1$ is
 a $\gk$Gorenstein, $\Hom_{\overline{R}}(R/{\fc_1},\overline{K})\cong\Ext^n_R(R/{\fc_1},K)\cong R/{\fc_1}$.
Therefore, for each $R/{\fc_1}$--module $N$,
\[\begin{array}{rl}
\Hom_{\overline{R}}(N,\overline{K})&\cong\Hom_{\overline{R}}(N\otimes_{\overline{R}}R/{\fc_1},\overline{K})\\
&\cong\Hom_{\overline{R}}(N,\Hom_{\overline{R}}(R/{\fc_1},\overline{K}))\\
&\cong\Hom_{R/{\fc_1}}(N,R/{\fc_1}).
\end{array}\]
Hence, we have the commutative diagram
$$\begin{CD}
&&&&&&&&\\
\ \ &&&& 0@>>>M@>f>>M^{\triangledown\triangledown}@>>>\coker(f)@>>>0&  \\
&&&&&& @VV{=}V @VV{\cong}V \\
\ \  &&&&0@>>>M@>e_M>>M^{**}@>>>\Ext^1_{R/{\fc_1}}(M_1,R/{\fc_1})@>>>0&\\
\end{CD}$$\\
with exact rows, where $M^*=\Hom_{R/{\fc_1}}(M,R/{c_1})$ and $M^{\triangledown}=\Hom_{\overline{R}}(M,\overline{K})$.
Thus, we have  $\coker(f)\cong\Ext^1_{R/{\fc_1}}(M_1,R/{\fc_1})$.
Similarly, $\coker(f)\cong\Ext^1_{R/{\fc_2}}(M_2,R/{\fc_2})$.

Now suppose that $i>1$. As $M_1\cong\lambda_{R/{\fc_1}}M$, it follows, from (\ref{2}) and Lemma \ref{G1}, that
 $\Omega_{R/{\fc_1}}M_1
\cong\Hom_{R/{\fc_1}}(M,R/{\fc_1})\cong\Ext^n_R(M,K)$. Therefore, one has
\[\begin{array}{rl}
\Ext^i_{R/{\fc_1}}(M_1,R/{\fc_1})&\cong\Ext^{i-1}_{R/{\fc_1}}(\Omega_{R/{\fc_1}}M_1,R/{\fc_1})\\
&\cong\Ext^{i-1}_{R/{\fc_1}}(\Ext^n_R(M,K),R/{\fc_1})\\
&\cong\Ext^{i-1+n}_R(\Ext^n_R(M,K),K),
\end{array}\]
where the last isomorphism is obtained by Lemma \ref{G1}.
By the same argument, one has $\Ext^i_{R/{\fc_2}}(M_2,R/{\fc_2})\cong\Ext^{i-1+n}_R(\Ext^n_R(M,K),K)$.
\end{proof}
\begin{cor}
Let $(R,\fm)$ be a Cohen-Macaulay local ring with canonical module $\omega_R$. Assume that $\fc_1$ and $\fc_2$ are
Gorenstein ideals and that
$M_1$,$M$, and $M_2$ are $R$--modules such that $M_1\underset{\fc_1}{\thicksim}M$ and $M\underset{\fc_2}{\thicksim}M_2$.
 Set \emph{$n=\dim_R(M_1)=\dim_R(M_2)$}. Then \emph{$\H^i_{\fm}(M_1)\cong\H^i_{\fm}(M_2)$}, for all $i$, $i<n$.
\end{cor}
\begin{proof}
By Proposition \ref{p4} and the Local Duality Theorem the assertion is obvious.
\end{proof}
We end the paper by another Corollary which in its first part we generalize \cite[Proposition 15]{MS}.
\begin{cor}\label{end}
Let $K$ be a semidualizing $R$--module and let $\fc_1$ and $\fc_2$ be two $\gk$Gorenstein ideals.
Assume that $M_1$,$M$, and $M_2$ are $R$--modules such that $M_1\underset{\fc_1}{\thicksim}M$ and $M\underset{\fc_2}{\thicksim}M_2$.
Then the following statements hold true.
\begin{itemize}
       \item[(i)]{\emph{$\gkd_R(M_1)=\gkd_R(M_2)$};}
        \item[(ii)]{If these dimensions are finite then $M_1$ satisfies \emph{$\widetilde{S}_k$} if and only if $M_2$
        satisfies \emph{$\widetilde{S}_k$}.}
\end{itemize}
\end{cor}
\begin{proof}
(i). By Theorem \ref{G2}, $\gkd_R(M_1)<\infty$ if and only if $\gd_{R/{\fc_1}}(M_1)<\infty$. As by (\ref{2}) and Lemma \ref{G1},
$\Omega_{R/{\fc_1}}M_1
\cong\Hom_{R/{\fc_1}}(M,R/{\fc_1})\cong\Ext^n_R(M,K)$, it follows that $\gkd_R(M_i)<\infty$ if and only if
$\gkd_R(\Ext^n_R(M,K))<\infty$ for $i=1, 2$.

Now, let $M_1$ and $M_2$ have
finite $\gk$dimensions. Set $n=\gr(\fc_1)=\gr(\fc_2)$. Note that, by Lemma \ref{l2}, $n=\gr(M_1)\leq\gkd_R(M_1)$.
If $n=\gkd_R(M_1)$ then $M_1$ is $\gk$perfect and so
the assertion is obvious by Proposition \ref{P3}. Now suppose that $\gkd_R(M_1)>n$. Hence by Theorem \ref{G3},
and Proposition \ref{p4} the assertion is obvious.

(ii). By Lemma \ref{G1}, $\Ext^i_{R/{\fc_1}}(M,R/{\fc_1})\cong\Ext^{n+i}_R(M,K)\cong\Ext^i_{R/{\fc_2}}(M,R/{\fc_2})$
for all $i\geq 0$ and so
 $\rgr_{R/{\fc_1}}(M)=\rgr_{R/{\fc_2}}(M)$. Hence the assertion is obvious by Corollary \ref{A}.
\end{proof}
{\bf Acknowledgement.} The authors thank J. R. Strooker and A. M.
Simon for their several valuable comments. Indeed they recognized
that the reduced grade has been defined first by Hyman Bass. They
also mentioned a flaw in one of our main result and we correct it in
its present form Theorem \ref{t}.

\bibliographystyle{amsplain}

\end{document}